\documentclass[12pt]{article}
\usepackage[a4paper,margin=3cm]{geometry}
\usepackage{graphicx,xcolor} 
\usepackage{amsmath,amssymb,amsthm}
\usepackage[shortlabels]{enumitem}
\usepackage{hyperref}
\usepackage{comment}

\newcommand{\F}{\mathbb{F}}
\newcommand{\N}{\mathbb{N}}
\newcommand{\Z}{\mathbb{Z}}

\newcommand{\PG}{\mathrm{PG}}
\newcommand{\Q}{\mathcal{Q}}

\newcommand{\Points}{\mathcal{P}}
\newcommand{\Lines}{\mathcal{L}}
\newcommand{\I}{\mathtt{I}}

\newcommand{\SCHex}{\mathsf{H}}

\newcommand{\FP}{\F\Points}
\newcommand{\Code}{\mathfrak{C}_\F}

\newcommand{\f}{\mathfrak{f}}
\newcommand{\g}{\mathfrak{g}}

\renewcommand{\i}{\mathfrak{i}}
\renewcommand{\c}{\mathfrak{c}}
\newcommand{\cx}{\mathfrak{c}_{x}}
\newcommand{\cv}{\mathfrak{c}_{v}}

\newtheorem{theorem}{Theorem}
\newtheorem{definition}{Definition}
\newtheorem{corollary}{Corollary}
\newtheorem{lemma}{Lemma}
\theoremstyle{remark}
\newtheorem{remark}{Remark}

\title{On certain blocking sets and the minimum weight of the code of generalised polygons}
\author{Sebastian Petit \thanks{School of Mathematics and Statistics, University of Canterbury, Private Bag 4800, 8140 Christchurch, New Zealand \href{mailto:sebastian.petit@pg.canterbury.ac.nz}{\url{sebastian.petit@pg.canterbury.ac.nz}}}
\and Geertrui Van de Voorde\thanks{School of Mathematics and Statistics, University of Canterbury, Private Bag 4800, 8140 Christchurch, New Zealand. ORCID: 0000-0002-4957-6911. \href{mailto:geertrui.vandevoorde@canterbury.ac.nz}{\url{geertrui.vandevoorde@canterbury.ac.nz}}}}
\date{}

\begin{document}

\maketitle
\begin{abstract}
In this paper, we study and characterise certain blocking sets in generalised polygons. This will allow us to derive new results about the minimum weight and minimum weight code words in the code generated by the rows of the incidence matrix of a  generalised polygon over a field \(\F\).
\end{abstract}


\section{Introduction}
\subsection{Motivation}
The study of the code generated by the incidence matrix of points and blocks of a design has a long and rich history, described in the monograph \cite{AssmusAndKey}. Arguably the best studied case is that of points and lines in a projective plane, partially motivated by the relevance of such codes in the proof of the non-existence of a projective plane of order 10 \cite{Lam}.

A classical result in this area shows that the $p$-ary code generated by the incidence vectors of lines of a projective plane of order $q=p^h$, $p$ prime, has minimum weight $q+1$, and the minimum weight vectors are precisely the scalar multiples of incidence vectors of lines (see Theorem \ref{AK}). An easy proof of this result is based on the correspondence between small weight code words in this code with blocking sets in the plane and the fact that the smallest blocking sets in the plane are given by lines (see Theorem \ref{BoseBurton}).

Similar results have been shown for the code generated by subspaces of $\PG(k,q)$ and for the code generated by lines of a regular generalised $2m$-gon. We will provide more details in Subsection \ref{subsection:codes}.

This paper contributes to this study by showing that the minimum weight of the code of points and lines in an arbitrary, not necessarily regular, thick generalised polygon is still the weight of a line. To prove this, we will establish a connection with certain blocking sets in weak generalised polygons, regardless of whether the polygon is regular or not. Our first main result (Theorem \ref{distanceTraces}) will find a lower bound on the size of such blocking sets, but it will also show that they behave in a fundamentally different way than in the case of projective spaces.
As a consequence of this difference in behaviour, some conditions will arise in our characterisation of minimum weight codewords in the code of not necessarily regular thick generalised polygons (see Theorem \ref{main}).

\subsection{Background}

\subsubsection{Weak generalised polygons}

We now formally introduce the notions used in this paper. As usual, a point-line geometry is a triple consisting of a point set $\Points$, a line set $\Lines$ and a symmetric incidence relation \(\I\subseteq(\Points\times\Lines) \cup (\Lines\times\Points)\). An {\em element} of a point-line geometry is a point or a line of that geometry.

\begin{definition}
	For \(n\ge 3\), a {\em weak generalised $n$-gon} is a point-line geometry satisfying:   
	\begin{enumerate}[(GP1)]
		\item There exists no \(k\)-gon (as a subgeometry) for \(2 \le k < n\).
		\item Every two elements of \(\Points \cup \Lines\) are contained in an \(n\)-gon.
	\end{enumerate}
\end{definition}

\begin{definition}
	A {\em thick generalised \(n\)-gon} is a weak generalised \(n\)-gon with the additional property:
	\begin{enumerate}[(GP3)]
		\item There exists an \((n+1)\)-gon (as a subgeometry).
	\end{enumerate}
\end{definition}

\begin{remark}\label{thickAlt}
    Property (GP3) is equivalent to the following, see for instance \cite{Hendrik}:
    \textit{
    \begin{enumerate}[(GP3')]
		\item Every point is incident with at least three lines and every line is incident with at least three points. 
	\end{enumerate}}
\end{remark}

\begin{remark} In the literature, the term \emph{generalised polygon} is sometimes used for what we call thick generalised polygons. To avoid ambiguity, we will always add \emph{weak} or \emph{thick}.\end{remark}

\begin{remark} Most of our results apply to weak generalised polygons. In our final result, Theorem \ref{main}, the other conditions we impose imply that the considered generalised polygon needs to be thick. 
\end{remark}

The {\em incidence graph} of a point-line geometry is the bipartite graph with as vertices the points and the lines and an edge between vertices \(x\) and \(y\) if and only if \(x \I y\).

A weak generalised \(n\)-gon can also be defined via its incidence graph:
\begin{corollary}(See e.g. \cite[Theorem 1.5.10]{Hendrik})\label{incidenceGraph}
	Let \(\Gamma\) be a point-line geometry with at least two points and the properties that:
	\begin{itemize}
		\item every point is incident with at least two lines;
		\item every line is incident with at least two points.
	\end{itemize}
	Then, \(\Gamma\) is a weak generalised \(n\)-gon if and only if its incidence graph has diameter \(n\) and girth \(2n\).
\end{corollary}

The {\em distance} between elements of a weak generalised $n$-gon $\Gamma$ is the distance inherited from its incidence graph. Two points, resp. lines, are {\em opposite} if they are at the furthest possible distance; the distance between opposite points, resp. lines, is given by $n$.

The following useful observation follows directly from the definition of a weak generalised polygon. 
\begin{corollary}(See e.g. \cite[Theorem 1.3.5]{Hendrik})
\label{uniqueClosest}
	Let \(p\) be a point and \(L\) be a line of a weak generalised \(2m\)-gon \(\Gamma\), $m\geq 2$. Then there is a unique point on \(L\) closest (that is, at smallest distance) to \(p\). 
\end{corollary}

We will only deal with weak finite generalised polygons that have an {\em order} $(s,t)$, that is, such that every line is incident with \(s+1\) points and every point is incident with \(t+1\) lines for some integers $s,t\geq 1$. It is well-known (see for example \cite[Corollary 1.5.3]{Hendrik}) that every thick finite generalised polygon has an order. While having an order is not a strong restriction, the famous theorem of Feit and Higman shows that there are rather strong restrictions known on the possible values of $n,s,t$:

\begin{theorem}(Feit \& Higman \cite{FEITHIGMAN})
	Let \(\Gamma\) be a finite weak generalised \(n\)-gon of order \((s,t)\) with \(n\ge 3\). Then one of the following holds:
	\begin{enumerate}[(i)]
		\item \(s=t=1\), and \(\Gamma\) is an ordinary \(n\)-gon;
		\item \(n=3,s=t>1\), and \(\Gamma\) is a projective plane;
		\item \(n = 4\) and \(\frac{st(1+st)}{s+t}\) is an integer;
		\item \(n=6\) and if \(s,t>1\), then \(st\) is a perfect square;
		\item \(n=8\) and if \(s,t>1\) then \(2st\) is a perfect square;
		\item \(n=12\) and \(s=1\) or \(t=1\). 
	\end{enumerate}
\end{theorem}
For thick generalised polygons, more is known:
\begin{theorem}
	Let \(\Gamma\) be a finite thick generalised \(n\)-gon of order \((s,t)\) with \(n\ge 4\). Then one of the following holds:
	\begin{enumerate}[(i)]
		\item (Higman \cite{Higman}) \(n = 4 \) and \(s\le t^2\), dually \(t\le s^2\);
		\item (Haemers \& Roos \cite{HaemersAndRoos}) \(n = 6\) and \(s\le t^3\), dually \(t\le s^3\);
		\item (Higman \cite{Higman}) \(n = 8\) and \(s\le t^2\), dually \(t\le s^2\).
		
	\end{enumerate}
\end{theorem}

Many non-isomorphic thick generalised quadrangles are known but up to duality, only two classes of finite thick generalised hexagons are known: the \emph{split Cayley hexagons} of order \((q,q)\) and the \emph{twisted triality hexagons} of order \((q^3,q)\). Only one class of finite thick generalised octagons is known (again up to duality): the Ree-Tits octagons of order \((q,q^2)\). 
For more information, we refer to \cite{Hendrik}.



\subsubsection{Codes}\label{subsection:codes}
In this paper, we will study the code  \(\Code(\Gamma)\), where \(\Gamma\) is a finite weak generalised \(2m\)-gon. This code will be defined using the following conventions.
\begin{definition}\label{def}
	Let \(\mathbb{F}\) be a field and let $\mathcal{P}$ be a set. 
    \begin{itemize}
    \item \(\mathbb{F}\Points\) is the vector space of \(\mathbb{F}\)-valued functions on the set \(\Points\); 
    \item For a subspace 
	 \(\mathfrak{M}\) of \(\mathbb{F}\Points\), its {\em dual}, \(\mathfrak{M}^D\), is the subspace of all \(\mathfrak{f}\in \FP:\) 
	\[\forall \mathfrak{g}\in\mathfrak{M}: \sum_{x\in\Points}\mathfrak{f}(x)\cdot\mathfrak{g}(x) = 0;\]
    \item	For a subset \(S\) of \(\Points\),  its {\em indicator function} \(\mathfrak{i}_{S} \in \mathbb{F}\Points\) is given by  \[\mathfrak{i}_{S}:\Points\to \mathbb{F}:p\mapsto \mathfrak{i}_{S}(p) = \begin{cases}
		1 &\text{ if } p\in S \\
		0 &\text{ if } p\not\in S
	\end{cases};\]
    \item For \(\mathfrak{f}\in\FP\),  the {\em support} of \(\mathfrak{f}\), denoted as \(\text{supp}(\mathfrak{f})\), is the set \(\{x\in\Points\mid \mathfrak{f}(x)\neq 0\}\); 
    \item The weight of a vector $\mathfrak{f}\in \FP$ is the size of $\text{supp}(\mathfrak{f})$.
    \end{itemize}
\end{definition}

To simplify notation, given a line \(L\), we write \(\mathfrak{i}_L\) for \( \mathfrak{i}_{\Gamma_{1}(L)}\) where \(\Gamma_1(L)\) is the set of points incident with the line \(L\) (see also Lemma \ref{GammaSets}). In other words, we identify a line with its point set. 

\begin{definition}
	The \(\mathbb{F}\)-ary code of the generalised polygon, denoted by \(\mathfrak{C}_{\mathbb{F}}(\Gamma)\), is the vector subspace of \(\FP\) generated by \(\{\i_L\mid L\in\Lines \}\). The {\em dual} code, denoted by \(\Code^D(\Gamma)\) is the dual of this subspace (as defined in Definition \ref{def}).
\end{definition}



The following results about \(\Code(\Gamma)\) and \(\Code^D(\Gamma)\) are known. 
\begin{theorem}(Bagchi \& Sastry \cite{BandS})\label{BS}
	Let \(\Gamma\) be a finite thick generalised \(2m\)-gon of order \((s,t)\) with \(m\ge2\). Let \(\F\) be any field. Then:
	\begin{enumerate}[(i)]
		\item The minimum weight of \(\Code^D(\Gamma)\) is at least \(2(t^m-1)(t-1)^{-1}\). 
		If \(\Gamma\) is \emph{regular}, then equality holds.
		\item The minimum weight of \(\Code(\Gamma)\) is at most \(s+1\). If \(\Gamma\) is \emph{regular}, then equality holds and any word of weight \(s+1\) is of the form \(\lambda\cdot\i_L\) for some \(0\neq\lambda\in \F\) and some line \(L\).
	\end{enumerate}
\end{theorem}

A weak generalised $2m$-gon of order $(s,t)$ is called {\em regular} if every pair of opposite points is contained in a (necessarily unique) $(1,t)$- subpolygon.

The following finite thick generalised \(2m\)-gons are regular:
\begin{itemize}
	\item The generalised quadrangles \(W(q)\), with \(q\) a prime power; these have order $(q,q)$.
	\item The generalised quadrangles \(H(3,q^2)\), with \(q\) a prime power; these have order $(q^2,q)$.
	\item The split Cayley hexagons \(\mathsf{H}(q)\), with \(q\) a prime power; these have order $(q,q)$.
	\item The twisted triality hexagons \(\mathsf{T}(q^3,q)\), with \(q\) a prime power; these have order $(q^3,q).$
	\item The Ree-Tits octagons \(\mathsf{O}(q,q^2)\), with \(q\) an odd power of two; these have order $(q,q^2).$
\end{itemize}

Unless \(\mathsf{H}(q)\) is self dual (which happens if and only if $q$ is a power of $3$), its dual is not regular. Likewise, the dual twisted triality hexagon is not regular. \\

The fact that the minimum weight of the code \(\Code(\Gamma)\) where $\Gamma$ has order $(s,t)$ is at most $s+1$ is easily seen by the fact that the incidence vector of a line is contained in the code and has weight $s+1$.
The proof that equality holds as well as the characterisation result of Theorem \ref{BS} (ii) in the regular case rely entirely on the existence of subpolygons of order $(1,t)$; these are shown to define code words in the dual code. A relatively easy argument then shows that a code word in the code of weight $s+1$ which needs to be orthogonal to all those subpolygon code words, corresponds to a line \cite[Lemma 2.5]{BandS}. Without the presence of subpolygons, the situation becomes a lot more difficult.
As mentioned before, in this paper, we will show that, under certain conditions, the minimum weight of \(\Code(\Gamma)\) is still \(s+1\) even if \(\Gamma\) is not regular and we will also describe the supports of code words with this minimum weight (see Theorem \ref{main}).

\subsubsection{\texorpdfstring{Thick generalised \(3\)-gons and embeddable weak generalised polygons}{Thick generalised 3-gons and embeddable weak generalised polygons}}
A thick generalised \(n\)-gon with \(n=3\) is a projective plane. The classical example of a finite projective plane is the Desarguesian projective plane \(\text{PG}(2,q)\) but many other non-Desarguesian planes are known. The study of their associated codes has been the source of great research interest, see \cite[Chapter 6]{AssmusAndKey}. The following result regarding the codes of projective planes is well known. We also include the result for the code arising from the geometry of points and lines in a higher dimensional projective space.

\begin{theorem}\label{AK}(Assmus \& Key \cite[Theorem 6.3.1 and Corollary 5.7.5]{AssmusAndKey})
	Let $\F=\mathbb{F}_p$, $p$ prime and let \(\Gamma\) be the geometry of points and lines in $\Pi$, where $\Pi$ is an arbitrary projective plane whose order $q$ is a multiple of $p$, or $\Pi$ is the projective geometry $\PG(k,q)$, $q=p^h$. Then the minimum weight of \(\Code(\Gamma)\) is \(q+1\) and the code words of minimum weight are scalar multiples of incidence vectors of lines.
\end{theorem}
\begin{remark}\label{remarkfield}
It is important to remark that the code \(\Code(\Gamma)\), where $\Gamma=\PG(k,q)$ is {\em trivial} when $\mathbb{F}$ is a finite field whose characteristic is different from $p$ (see e.g. \cite[Theorem 2.5]{SurveyCodes}); in this context, trivial means that the code is either the full space or the dual of the all-one vector. Note that the minimum distance of a trivial code is $1$ or $2$ respectively, which is strictly smaller than $q+1$, the weight of the incidence vector of a line.
\end{remark}

	We say that a point-line geometry \(\Gamma = (\Points,\Lines,\I)\) is {\em fully embedded} in another point-line geometry $\Gamma' = (\Points',\Lines',\I')$ if
	\(\Points\subseteq\Points'\),		\(\Lines\subseteq\Lines'\),
		and for all lines \(L\) in \(\Lines\) all points of \(\Points'\) on \(L\) are in \(\Points\).
The following statement now easily follows from Theorem \ref{AK}.

\begin{corollary}\label{cor:embedded}
	Let \(\Gamma\) be a weak generalised \(n\)-gon which is fully embedded in $\PG(k, q)$ and let \(\F\) be the field \(\F_p\) where \(p\) is a divisor of \(q\). Then the minimum weight of \(\Code(\Gamma)\) is \(q+1\).
\end{corollary}
\begin{proof}
	The code \(\Code(\Gamma)\) is a subspace of the code \(\Code(\Gamma')\), where $\Gamma'$ is the geometry of points and lines from $\PG(k,q)$. It follows from Theorem \ref{AK} that \(q+1\) is a lower bound for the minimum weight of \(\Code(\Gamma)\), and the incidence vector of any line of \(\Gamma \) shows equality.
\end{proof}

\begin{remark} Corollary \ref{cor:embedded} has two strong requirements which make its use rather limited: not all weak generalised polygons are embeddable in a projective space $\PG(k,q)$, and even for those who are embeddable, Corollary \ref{cor:embedded} doesn't say anything about the case where $\mathbb{F}$ is different from $\mathbb{F}_p$. It can be seen from Theorem \ref{BS} that it is not because the code \(\Code(\PG(k,q))\) is trivial for $\mathbb{F}\neq \mathbb{F}_p$ that its subcode \(\Code(\Gamma)\) is trivial (if it would, the minimum weight would be at most $2$). Furthermore, one should also not deduce from Theorem \ref{AK} that all minimum-weight vectors of \(\Code(\Gamma)\)  necessarily correspond to lines of $\Gamma$; it only follows that the minimum weight vectors correspond to point sets arising from a line of the ambient projective space.
\end{remark}

\subsubsection{Blocking sets in projective space}
A {\em blocking set} with respect to lines in a projective plane is a set of points such that every line contains at least one point of $B$. The study of blocking sets in projective planes dates back to the early 1950s \cite{Richardson} and the study of blocking sets in projective spaces has attracted a lot of attention in finite geometry (see e.g. \cite{blokhuis2011blocking}). 

The fundamental Bose-Burton theorem can be phrased as follows:

\begin{theorem}(Bose \& Burton \cite{BoseBurton})\label{BoseBurton}
A set $B$ of points of $\PG(k,q)$ blocking all hyperplanes has at least the size of a line, and equality holds if and only if the set $B$ is the point set of a line.
\end{theorem}

Theorem \ref{BoseBurton} provides a combinatorial characterisation of a line: (the point set of) a line blocks the set $\mathcal{H}$ of all hyperplanes, and this theorem shows that all $\mathcal{H}$-blocking sets of this size are given by lines.
It is a natural question to consider the same problem for other geometries. We consider (the point set of) a line, find out which set $\mathcal{X}$ of structures is blocked by this point set, and then determine whether or not the only possible $\mathcal{X}$-blocking sets of this size are given by lines.
We will see in Theorem \ref{distanceTraces} that the answer to the latter question for generalised $2m$-gons is, in general, not necessarily yes, contrasting the case of thick generalised $3$-gons.


\section{Blocking sets in weak generalised polygons}\label{section:BlockingSet}

We have seen in Theorem \ref{AK} (and Remark \ref{remarkfield}) that the problem of finding the minimum weight of the code of generalised $3$-gons has been settled, so from now on we will let  \(\Gamma = (\Points, \Lines, \I)\)  denote a weak generalised \(2m\)-gon of order \((s,t)\), with \(m\in \N \ge 2\), $s,t\geq 1$. 
\begin{remark}\label{possibleZero}
	Let \(\F\) be a field and let \(k\) be a natural number. We see that
	\(k\cdot 1_\F := 1_\F+1_\F+\dots+1_\F \text{ (\(k\) times)}\) and \( -k\cdot 1_\F := -(k\cdot 1_\F) \) are elements of \(\F\). This allows us to evaluate any integer in \(\F\) and, with some abuse of notation, use integers as parameters in \(\F\). Note however, that it is possible for this parameter to evaluate to \(0\) in \(\F\).
\end{remark}

Finally, since we will only be concerned with the code of the $2m$-gon \(\Gamma = (\Points, \Lines, \I)\), we will omit $\Gamma$ and use \(\mathfrak{C}_{\mathbb{F}}\) for \( \mathfrak{C}_{\mathbb{F}}(\Gamma)\), the vector subspace of \(\FP\) generated by \(\{\i_L\mid L\in\Lines \}\).

\subsection{Some elementary geometric properties}

\begin{definition}\label{GammaSets}\label{def:GammaSets}
	Let $\Gamma=(\Points,\Lines, \I)$ be a point-line geometry and let $x$ be an element of $\Gamma$. We define the following sets:
	\begin{align*}
		\Gamma_i(x) &:= \{y \in \Points\cup \Lines : \delta(x,y) = i\}, \\
		\Gamma_{\le i}(x) &:= \{y \in \Points\cup\Lines: \delta(x,y) \le i\}, \\
		\Points_i(x) &:= \{y \in \Points: \delta(x,y) = i\}, \\
		\Points_{\le i}(x) &:= \{y \in \Points: \delta(x,y) \le i\}, \\ 
		\Lines_i(x) &:= \{y \in \Lines: \delta(x,y) = i\}, \\
		\Lines_{\le i}(x) &:= \{y \in \Lines: \delta(x,y) \le i\} ,
	\end{align*}
	where \(\delta(x,y)\) is the distance between \(x\) and \(y\) (inherited from the incidence graph).
\end{definition}
Note that if $i\leq 0$, the sets in Definition \ref{GammaSets} are still defined, but are empty. The following easily follows from the definition.
\begin{corollary}\label{gammaProperties}
	\begin{enumerate}[(i)]
		\item For a point \(x\in\Points\), \(\Points_i(x) = \emptyset\) if $i$ is odd, and   \(\Points_i(x) = \Gamma_i(x)\) if $i$ is even.
		\item For a line \(L\in\Lines\), \(\Points_i(L) = \Gamma_i(L)\) if $i$ is odd and \(\Points_i(L) = \emptyset\) if $i$ is even. 
	\end{enumerate}
\end{corollary}

We will now investigate the possible intersection of some sets introduced in Definition \ref{def:GammaSets}. In general, we have the following:
\begin{lemma}\label{intersections1}
	Let \(x\) be an element of $\Gamma$. Let $y_1, y_2$ be elements with $\delta(x,y_1)=\delta(x,y_2)=d_1$ and $\delta(y_1,y_2)=2d_1$ with \(d_1\in\{1,\dots,m-1\}\). 
    
    Let \(d_2\in\{1,\dots,2m-d_1-1\}\).
	Then,
	\[\Gamma_{\le d_2-d_1}(x) = \Gamma_{\le d_2}(y_1)\cap\Gamma_{\le d_2}(y_2).\]
\end{lemma}
\begin{proof}
	We clearly have that
	\[\Gamma_{\le d_2-d_1}(x) \subseteq \Gamma_{\le d_2}(y_1)\cap\Gamma_{\le d_2}(y_2).\]
	Now assume that there is an element \(z\) in \(\Gamma_{\le d_2}(y_1)\cap\Gamma_{\le d_2}(y_2)\) with \(z\not\in\)  \(\Gamma_{\le d_2-d_1}(x)\). Then, combining the paths from \(z\) to \(y_1\), from \(y_1\) to \(y_2\) and from \(y_2\) back to \(z\), we find a cycle of length at most \(d_2 + 2d_1+ d_2 \le 4m-2d_1-2+2d_1 = 4m-2.\) This contradicts Theorem \ref{incidenceGraph} which states that the incidence graph of a \(2m\)-gon has girth \(4m\).
\end{proof}

As a corollary, we determine what happens in some specific cases.
\begin{corollary}\label{intersections2}
\begin{enumerate}[(i)]
\item 	Let \(x\) be an element of $\Gamma$. Let \(y_1\neq y_2 \in \Gamma_1(x)\) and let \(d\in\{1,\dots,2m-2\}\).
	Then,
	\[\Gamma_{\le d - 1}(x) = \Gamma_{\le d}(y_1)\cap\Gamma_{\le d}(y_2).\]

\item  Let \(L\) be a line of \(\Gamma\) and let \(r\in \{1,2,\dots,m-1\}\).
	\begin{enumerate}[(a)]
		\item If 
		\(M_1\) and \(M_2\) are two disjoint lines intersecting \(L\),
		then
		\[\Points_{\le 2r-3}(L) = \Points_{\le 2r-1}(M_1)\cap\Points_{\le 2r-1}(M_2).\]
		\item If
		\(v_1\) and \(v_2\) are two distinct points on \(L\), then
		\[\Points_{\le 2r-1}(L) = \Points_{\le 2r}(v_1)\cap\Points_{\le 2r}(v_2).\]
	\end{enumerate}
\end{enumerate}
\end{corollary}
\begin{proof}
\begin{enumerate}[(i)]
\item 	This follows from Lemma \ref{intersections1} with \(d_1 = 1\).
\item \begin{enumerate}[(a)]\item We have that $\delta(L,M_1)=\delta(L,M_2)=2$ and $\delta(M_1,M_2)=4$. Lemma \ref{intersections1} with $d_1=2$ and  $d_2=2r-1$ shows that $$\Gamma_{\leq 2r-3 }(L)=\Gamma_{\leq 2r-1}(M_1)\cap \Gamma_{\leq 2r-1}(M_2).$$
Restricting these sets to their intersection with $\mathcal{P}$ yields the desired result.
\item This follows from noting that $\delta(v_1,v_2)$ is even and restricting the equality found in part 1 to the point set $\Points$.
\end{enumerate}
\end{enumerate}
\end{proof}



The following lemma will be used in Subsection \ref{subsection:linesasbs}. 
\begin{lemma}\label{starLemma}
	Let \(C\) be a set of points of $\Gamma$ with \(|C|<s+1\). Let \(d\in \{1,2,\dots,m-1\}\).
	\begin{enumerate}[(i)]
		\item \(\forall L \in \Lines: \exists M \in \Lines_{\le 2d-2}(L): C\cap \Points_{\le2d-1}(M) = C\cap \Points_{1}(L)\).
		\item \(\forall L \in \Lines: \exists v \in \Points_{\le 2d-1}(L): C\cap \Points_{\le2d}(v) = C\cap \Points_{1}(L)\).
	\end{enumerate}
\end{lemma}
\begin{proof}
	\(\)
	\begin{enumerate}[(i)]
		\item We prove this by induction on \(d\). The base case \(d=1\) follows by taking $M=L$.
		Assume that the statement holds for some \(k\in \N\) with \(1\le k < m-1\) and consider the case \(d=k+1\). 
		Let \(L\in\Lines\) be an arbitrary line of \(\Gamma\). Then there exists an \(M\in \Lines_{\le 2k-2}(L)\) such that \(C\cap \Points_{\le2k-1}(M) = C\cap \Points_{1}(L)\). 
		
		We label the \(s+1\) points on \(M\) as \(\{v_1,v_2,\dots,v_{s+1}\}\). For each of these points \(v_i\) we can consider a line \(N_i\) through \(v_i\) and different from \(M\). Note that the lines $N_1,\ldots,N_{s+1}$ are mutually disjoint.
		
		By Corollary \ref{intersections2}(2a), we now observe that for any two of these lines \(N_i\) and \(N_j\) with \(i\neq j\), we have \[\Points_{\le 2(k+1)-1}(N_i)\cap \Points_{\le 2(k+1)-1}(N_j)= \Points_{\le 2(k+1)-3}(M).\]
		
		Since  \(|C|<s+1\), at least one of the $s+1$ disjoint sets \(\Points_{\le 2(k+1)-1}(N_i)\setminus \Points_{\le 2(k+1)-3}(M) \), $i \in \{1,2,\ldots,s+1\}$, must have an empty intersection with \(C\). Without loss of generality, we may assume 
		\[(\Points_{\le 2(k+1)-1}(N_1)\setminus \Points_{\le 2(k+1)-3}(M)) \cap C = \emptyset.\]
		This means that 
		\[C\cap \Points_{\le2(k+1)-1}(N_1) = C\cap \Points_{\le2(k+1)-3}(M) = C\cap \Points_{1}(L),\]
		 where we have invoked the induction hypothesis.
         
		Since \(M\in \Lines_{2k-2}(L)\) and \(N_1\in\Lines_{2}(M)\),  we see that \(N_1\in   \Lines_{2k-2+2}(L)=\Lines_{2k}(L)=\Lines_{2(k+1)-2}(L)\). The statement follows.
		
		\item Let \(L\in\Lines\) be an arbitrary line of \(\Gamma\). 
		From part (i), we know that there exists a line \(M\in \Lines_{\le 2d-2}(L)\) such that \(C\cap \Points_{\le2d-1}(M) = C\cap \Points_{1}(L)\).
		
		We label the \(s+1\) points on \(M\) as \(\{v_1,v_2,\dots,v_{s+1}\}\). 

		By Corollary \ref{intersections2}(2b), we have that for \(i\neq j\),  \[\Points_{\le 2d}(v_i)\cap \Points_{\le 2d}(v_j)= \Points_{\le 2d-1}(M).\]
		
		Since \(|C|<s+1\) points, at least one of the disjoint  sets \(\Points_{\le 2d}(v_i)\setminus \Points_{\le 2d-1}(M) \) must have an empty intersection with \(C\). Without loss of generality, we may assume 
		\[(\Points_{\le 2d}(v_1)\setminus \Points_{\le 2d-1}(M)) \cap C = \emptyset.\]
		This means that 
		\[C\cap \Points_{\le2d}(v_1) = C\cap \Points_{\le2d-1}(M) = C\cap \Points_{1}(L)\]
		
		Since $v_1\in \Points_1(M)$ and $M\subseteq \Lines_{\leq 2d-2}(L)$, we see that \(v_1\in \Points_{1}(M) \subseteq \Points_{\leq 2d-1}(L)\) and the statement follows.\qedhere
	\end{enumerate}
\end{proof}

\subsection{\texorpdfstring{Lines as $\mathcal{X}$-blocking sets in weak generalised polygons}{Lines as blocking sets in weak generalised polygons}}\label{subsection:linesasbs}

Let $\mathcal{S}$ be a set whose elements are subsets of points of a weak generalised polygon $\Gamma$. An {\em $\mathcal{S}$-blocking set} in $\Gamma$ is a set of points $B$ such that for all $S\in \mathcal{S}$, $B\cap S\neq \emptyset.$ Using this notation, we see that a blocking set with respect to lines in $\Gamma$ is a $\mathcal{T}$-blocking set where $\mathcal{T}=\{\mathcal{P}_1(L)\mid L\in \Lines\}$.

From now on, let $\mathcal{X}=\{\mathcal{P}_{\leq 2m-2}(v)\mid v \in \Points\}$. 
The following lemma shows that (the point set of) a line $L$ is an $\mathcal{X}$-blocking set.

\begin{lemma}\label{G4s}
	Let \(v\) be a point of \(\Gamma\). Then, every line \(L\) intersects \(\Points_{\le 2m-2}(v)\) in either \(1\) or \(s+1\) points.   
    
\end{lemma}
\begin{proof}
	By Corollary \ref{uniqueClosest}, \(L\) contains a unique point \(w\) closest to \(v\) and we know that \(\delta(v,w) \in \{0,2,\dots, 2m-2\}\).
	If \(\delta(v,w) < 2m-2\) then clearly all points on \(L\) are in \(\Points_{\leq 2m-2}(v)\). If \(\delta(v,w) = 2m-2\) then all points on \(L\) different from $w$ are at distance \(2m\) from \(v\) and therefore not in \(\Points_{\leq 2m-2}(v)\).
\end{proof}

The following lemma derives a lower bound on the size of a blocking set with respect to lines in $\Gamma$.
\begin{lemma}\label{smallC1} Let \(B\) be a set of points of \(\Gamma\) such that
	\[\forall L\in \Lines: |\Points_{1}(L)\cap B|\ge 1,\]
	then \(|B|\ge \frac{s^m t^m-1}{st-1}\).
\end{lemma}
\begin{proof}
	We double count the following set:
	\[X := \{(v,L)\mid v\in B, L\in \Lines, v\I L\}.\]
	Each point in $B$ lies on \(t+1\) lines so we find \(|X|=|B|(t+1)\).
	Because there are \(|\Lines| = (1+t)\frac{s^m t^m-1}{st-1}\) lines (see for instance \cite{Hendrik}), each of which needs to contain at least one point of  \(B\) we find \(|X|\ge(1+t)\frac{s^m t^m-1}{st-1} \).
	Combining these results gives us \(|B|\ge \frac{s^m t^m-1}{st-1} \).
\end{proof}

We have seen that lines are $\mathcal{X}$-blocking sets where $\mathcal{X}=\{\mathcal{P}_{\leq 2m-2}(v)\mid v \in \Points\}$; we now show that $\mathcal{X}$-blocking sets have at least the size of a line.

\begin{lemma}\label{smallC}
	Let \(C\) be an $\mathcal{X}-$blocking set,
	then \(|C|\ge s+1\).
\end{lemma}
\begin{proof}
	Assume to the contrary that \(|C|<s+1\). 
	Lemma \ref{starLemma} (ii) with \(d=m-1\) tells us that 
	\[\forall L \in \Lines: \exists x \in \Points_{\le 2m-3}(L):   C\cap \Points_{\le2m-2}(x)=C\cap \Points_{1}(L) .\]
	Since $\forall x$, $C\cap \Points_{\le 2m-2}(x)\neq\emptyset$ , we deduce that
	\[\forall L \in \Lines: C\cap \Points_{1}(L)\neq \emptyset.\]
	Using Lemma \ref{smallC1}, we find that $|C|\geq \frac{s^mt^m-1}{st-1}\geq s+1$, a contradiction.
\end{proof}

\subsection{\texorpdfstring{Characterising $\mathcal{X}$-blocking sets}{Characterising X-blocking sets}}\label{subsection:characterisation}
The goal of this subsection is to prove Theorem \ref{distanceTraces}, which characterises $\mathcal{X}$-blocking sets of minimum size, where, as before, $\mathcal{X}=\{\mathcal{P}_{\leq 2m-2}(v)\mid v \in \Points\}$. We will prove this using a series of lemmas.
\begin{lemma}\label{exists0} Assume that $s\leq t$. Let \(C\) be a set of points of size \(s+1\) and let \(x\) be an element of \(\Gamma\). Let \(d\in\{1,\dots,2m-2\}\).
	If \(\Points_{\le d-1}(x)\cap C = \emptyset\) and 
	\[\exists y \in \Gamma_1(x): |\Points_{\le d}(y)\cap C| \ge 2\] then
	\[\exists z \in \Gamma_1(x): \Points_{\le d}(z)\cap C = \emptyset. \]
\end{lemma}
\begin{proof}

	Let \(z_1\) and \(z_2\) be two different elements in \( \Gamma_1(x)\).
	By Corollary  \ref{gammaProperties} and \ref{intersections2}(1),
	\[\Points_{\le d}(z_1) \cap \Points_{\le d}(z_2)\cap C = \Points_{\le d-1}(x) \cap C= \emptyset.\] 
	Assume to the contrary that 
	\[\forall z \in \Gamma_1(x): |\Points_{\le d}(z)\cap C| > 0, \]
then we find that \(|C|\ge |\Gamma_1(x)\setminus\{y\}| + 2 \).

There are either $s$ (when $x$ is a line) or $t$ (when $x$ is a point) elements in $\Gamma_1(x),$ different from $y$. 
    Hence, since  \(s\le t\) there are at least \(s\) elements in \(\Gamma_1(x)\) different from \(y\). 
	We find \(|C|\ge |\Gamma_1(x)\setminus\{y\}| + 2 \ge s+2\), a contradiction since $|C|=s+1$.
\end{proof}

Note that in the following lemma, we use \(\Points_d(x)\) and \(\Points_{d-1}(y)\) and not \(\Points_{\le d}(x)\) and \(\Points_{\le d-1}(y)\).
\begin{lemma}\label{traceLines}
	
	Assume that \(s\le t\). Let $C$ be an $\mathcal{X}-$blocking set.
	Let \(x\) be an element of \(\Gamma\) such that for some \(d\in\{1,\dots,2m-2\}\):
	\begin{itemize}
		\item \(\Points_d(x)\) contains at least two points of \(C\).
		\item \(\Points_{\le d-1}(x)\) contains no points of \(C\).
	\end{itemize}
Then, $C$ is a $\mathcal{Y}-$blocking set where $\mathcal{Y}=\{\Points_{d-1}(y)\mid y\in \Gamma_1(x)\}.$

\end{lemma}
\begin{proof}
	
	Assume to the contrary that there is an element \(y\) in \(\Gamma_{1}(x)\) with 
	\[\Points_{d-1}(y)\cap C = \emptyset.\]
	We prove the following claim by induction on \(d'\):
	\[\forall d' \in \{d-1,\dots,2m-2\}:\exists z \in \Gamma_{\le d'-d +1}(y): \Points_{\le d'}(z)\cap C = \emptyset.\]
	The base case \(d' = d-1\) follows by considering the point \(y\): since \(\Gamma_{\le d-2}(y)\) is a subset of \(\Points_{\le d-1}(x)\), the statement follows from \(\Points_{\le d-1}(x)\cap C = \emptyset\) and \(\Points_{d-1}(y)\cap C = \emptyset\).
	
	Now assume that the statement holds for \(d' = k\) with \(d-1\le k<2m-2\). 
	We find an element \(z_k\) in \(\Gamma_{\le k-d+1}(y)\) such that \[\Points_{\le k}(z_k)\cap C = \emptyset.\] 
	
	Since \(z_k\in \Gamma_{\le k-d+1}(y)\) we also get \(y \in \Gamma_{\le k-d+1}(z_k)\).
	Hence, there is an element \(y'\) in \(\Gamma_{1}(z_k)\) such that \(y \in \Gamma_{\le k-d}(y')\). This implies that \[\Points_d(x) \subseteq \Points_{\le d+1}(y) \subseteq \Points_{\le k+1}(y').\]
	Therefore, \(\Points_{\le k+1}(y')\) contains at least two points of \(C\).
	We are now ready to use Lemma \ref{exists0} with \(x=z_k\) and \(y=y'\). This implies the existence of an element \(z_{k+1}\) in \(\Gamma_1(z_k)\subseteq \Gamma_{\le k-d+2}(y)\) such that \[\Points_{\le k+1}(z_{k+1})\cap C = \emptyset.\]
	This proves our claim.
	By setting \(d'=2m-2\), we now find a contradiction with the assumption that $C$ is an $\mathcal{X}-$blocking set.
\end{proof}

\begin{lemma}\label{pointsAtDist} Assume that \(s\le t\). 
	Let \(C\) be a set of points of size \(s+1\). 
	Let \(x\) be an element of \(\Gamma\) such that for some \(d\in\{1,\dots,2m-2\}\) the following hold:
	\begin{itemize}
		\item \(\Points_{\le d-2}(x)\) contains no points of \(C\).
		\item $C$ is a $\mathcal{Y}$- blocking set where $\mathcal{Y}=\{\Points_{d-1}(y)\mid y\in \Gamma_1(x)\}.$
        
	\end{itemize}
	Then, \(C\subseteq \Points_d(x)\) and if \(x\) is a point, then \(s=t\).
\end{lemma}
\begin{proof}
Corollary \ref{intersections2}(1) shows that for any two elements \(y_1\) and \(y_2\) in \(\Gamma_{1}(x)\) \[ \Gamma_{\le d-1}(y_1)\cap\Gamma_{\le d-1}(y_2)=\Gamma_{\le d - 2}(x).\]
This implies that 
\[\Points_{d-1}(y_1)\cap\Points_{d-1}(y_2) \subseteq \Points_{\leq d-2}(x).\]
Since $\mathcal{P}_{\leq d-2}\cap C=\emptyset$ we see that for every $y\in \Gamma_1(x),$ the set \(\Points_{d-1}(y)\setminus\Points_{d-2}(x)\) contains a different point of \(C\), and the points obtained in this way are mutually distinct. We find that $s+1=|C|\geq |\Gamma_1(x)|$. If $x$ is a point, $|\Gamma_1(x)|=t+1$, and by our assumption, $t\geq s$, so it follows that in this case $s=t$.
Furthermore, it follows that in both the case where $x$ a point or $x$ is a line, each of the sets  \(\Points_{d-1}(y)\setminus\Points_{d-2}(x)\) contains exactly one point of \(C\), and $C$ does not contain any further points. It follows that $C\subseteq \Points_{\leq d}(x)$, and since $\Points_{\leq d-2}(x)\cap C=\emptyset$, $C\subseteq \Points_d(x)$.
\end{proof}

\begin{definition}
	A {\em distance $d$-trace} \(T_{d,x,y}\), where \(d\in \{1,\dots,m\}\), is a non-empty set of points such that there exist opposite elements \(x\) and \(y\) in \(\Points\cup\Lines\) with 
	\[T_{d,x,y} = \Points_{d}(x)\cap\Points_{2m-d}(y).\]
\end{definition}
If we do not want to specify the distance $d$, we also call a distance $d$-trace a {\em distance trace}.
\begin{remark} Note that $x$ and $y$ play a different role in the definition so unless $d=m$, the distance $d$-trace determined by $x$ and $y$ is different from the distance $d$-trace determined by $y$ and $x$.
\end{remark}

\begin{remark} The point set of a line is an example of a distance trace: take $d=1$, $x=L$ and $M$ any line opposite to $L$, then $T_{1,L,M}=\Points_1(L)\cap \Points_{2m-1}(M)=\Points_1(L).$
\end{remark}
\begin{theorem}\label{distanceTraces}Assume that \(s\le t\).
Let $C$ be an $\mathcal{X}$-blocking set. Then $|C|\geq s+1$. If equality holds then $C$ is a distance $d$-trace $T_{d,x,y}$ for some opposite elements $x,y$ of $\Gamma$ and some $d\in \{1,\ldots,m\}$. 
 Furthermore, if equality holds and $s<t$, then $d$ is odd.
\end{theorem}
\begin{proof}
	Let \(z_1\neq z_2\) be two points of \(C\) such that $\delta(z_1,z_2)\leq 
    \delta(u,v)$ for all $u\neq v$ points of $C$.
    
    By the axioms of the generalised \(2m\)-gon $\Gamma$, there exists an ordinary \(2m\)-gon $\Gamma'$ containing \(z_1\) and \(z_2\). Since \(z_1\) and \(z_2\) are both points, the distance between them is even, and we can find opposite elements \(x\) and \(y\) in $\Gamma'$ such that:
	\begin{itemize}
		\item \(z_1,z_2 \in \Points_d(x)\)
		\item \(z_1,z_2 \in \Points_{2m-d}(y)\)
	\end{itemize}
	for \(d=\frac{\delta(z_1,z_2)}{2}\).
	
If \(\Points_{\le d-1}(x)\) would contain a point of \(C\), then \(z_1\) and \(z_2\) would not have been at the minimal distance. We can therefore use Lemma \ref{traceLines} and Lemma \ref{pointsAtDist} to conclude that all points of \(C\) are contained in \(\Points_{d}(x)\).
Since \(x\) and \(y\) are opposite, \(\Points_{\le d}(x)\cap \Points_{\le 2m-d-1}(y) = \emptyset\) and we get that \(\Points_{\le 2m-d-1}(y)\) can not contain any point of \(C\). Again, from lemma \ref{traceLines} and lemma \ref{pointsAtDist}, we can conclude that \(C\subseteq \Points_{2m-d}(y)\). 
We now have \(C\subseteq \Points_{d}(x)\cap\Points_{2m-d}(y)\) proving the statement.
\end{proof}

\begin{remark}\label{rem:traces} As seen before, when $d=1$, we consider opposite lines $x,y$, and $T_{1,x,y}$ is the set of points on the line $x$. 
In the case of weak generalised quadrangles, the distance $2$-trace is the set of points collinear to two non-collinear points. When the generalised quadrangle is $W(3,q)$, the distance $2$-trace is given by the points of a line of the ambient projective space (see also Remark \ref{remark:W3q}), and in the case of $\Q(4,q)$, this set is the point set of a conic in the ambient space. For the generalised hexagon $\SCHex(q)$ a similar situation arises. It is well-known that there is an embedding of \(\SCHex(q)\) in $\Q(6,q)$, which is a polar space embedded in $\PG(6,q)$.
This embedding has the property that for two points $v,w$ in $\SCHex(q)$, $\delta(v,w)=4$ if and only if $v$ and $w$ are collinear in $\Q(6,q)$ but not in $\SCHex(q)$.
This allows us to describe the distance $d$-traces in $\SCHex(q)$ geometrically: when $d=2$, we consider opposite points $x,y$. We find that $T_{2,x,y}$ is the set of $q+1$ points on a line of $\Q(6,q)$ that is not a line of $\SCHex(q)$. The set of points $\Points_{\leq 4}(v)$ is contained in the hyperplane $v^\perp$, where $\perp$ denotes the polarity corresponding to $\Q(6,q)$. The line $T_{2,x,y}$ and the hyperplane $v^\perp$ in $\PG(6,q)$ always intersect in at least a point, which is contained in $\Q(6,q)$, and hence, in $\SCHex(q)$. We deduce that in this case, $T_{2,x,y}$ forms an $\mathcal{X}$-blocking set.
When $d=3$, we consider again opposite lines $x,y$. The set of points at distance $3$ from both $x$ and $y$ is the set of $q+1$ points of $x^\perp\cap y^\perp$, where $\perp$ denotes the polarity corresponding to $\Q(6,q)$. These $q+1$ points form a conic in $\PG(6,q)$. 
\end{remark}

\begin{remark}
    The complement of the set $\Points_{\leq 2m-2}(v)$ is precisely the set of points at distance $2m$ of $v$, the set of points opposite to $v$. It follows that an $\mathcal{X}$-blocking set is a set of points $B$ such that there is no point opposite all of the points of $B$. 
    This point of view can be used to derive similar results for other point-line incidence geometries and to introduce the related concept of a  {\em geometric} line (see \cite{Kasikova}): this is a set of points such that each point is not opposite all or one point of the set. Note though that the definition of an $\mathcal{X}$-blocking set does not have fixed intersection sizes.
    \end{remark}
\subsection{Regularity and distance traces}
We observed in Theorem \ref{distanceTraces} that if $C$ is an $\mathcal{X}$-blocking set, then $C$ is a distance trace. The natural question arises whether every distance trace is an $\mathcal{X}$-blocking set. We have seen that for point sets of lines (which are examples of distance traces), this is the case. Similarly, in remark \ref{rem:traces} we gave an example in $\SCHex(q)$. We will now show more generally that if $\Gamma$ satisfies a certain regularity condition, the distance $2$-traces can form $\mathcal{X}$-blocking sets.

We need the definition of the perp geometry (see also \cite[1.9.2]{Hendrik}).

\begin{definition}
	Let \(x\) be a point of \(\Gamma\). The {\em perp-geometry} of a point \(x\) in \(\Gamma\) is the point-line geometry \((\Points',\Lines',\I')\) with
	\begin{itemize}
		\item \(\Points' = \Points_{ 2}(x)\),
		\item \(\Lines' = \Lines_{ 1}(x) \cup \{T_{2,x,y}\mid y\in\Points, y\text{ is opposite } x\}\),
		\item \(\I'\) defined naturally from \(\I\) and \(\in\).
	\end{itemize}
\end{definition}

A point $x$ is called {\em projective} if the perp-geometry at $x$ determines a projective plane. 

\begin{lemma}
    Let $x$ be a projective point of $\Gamma$, let $y$ be a point which is opposite to $x$ and let $T=T_{2,x,y}$ be the distance $2$-trace determined by $x$ and $y$. Then $T$ is an $\mathcal{X}$-blocking set.
\end{lemma}
\begin{proof}
    Let $\Gamma'=(\Points',\Lines',I')$ be the projective plane which is the perp-geometry determined by the projective point $x$. We need to show that for all $v\in \Points$, $\Points_{\leq 2m-2}(v)\cap T_{2,x,y}\neq \emptyset$. If $v$ is opposite $x$, then $T_{2,x,v}$ and $T_{2,x,y}$ are lines of $\Gamma'$, and since $\Gamma'$ is a projective plane, these two lines have a point $w$ of $\Gamma'$ in common. Since $w$ is a point of $T_{2,x,v}$, it follows that $w\in \Points_{\leq 2m-2}(v)\cap T_{2,x,y}.$
    
    If $v$ and $x$ are not opposite, $\delta(v,x)\leq 2m-2$. The second point of a shortest path from $x$ to $v$, say $r$, is contained in one of the lines through $x$, say $L$. We have that $r$ is contained in $\Points_{\leq 2m-4}(v)$, and since $L$ is a line of $\Gamma'$, it meets $T_{2,x,y}$, which is a line of $\Gamma'$ too, in a point $s$ of $\Gamma'$. Since $\delta(r,s)=2$, and $r\in \Points_{\leq 2m-4}(v)$ it follows that $s\in \Points_{\leq 2m-2}(v)\cap T_{2,x,y}.$ 
    We conclude that $T_{2,x,y}$ is an $\mathcal{X}-$ blocking set.
\end{proof}

\begin{remark} \label{remark:W3q} For certain families of weak generalised polygons, the existence of projective points is very easy to deduce. Consider for example the generalised quadrangle $W(3,q)$ which is embedded in $\PG(3,q)$, let $x$ be a point of $W(3,q)$ and let $\perp$ denote the associated polarity. The perp-geometry of $x$ is simply the point-line geometry of the projective plane $\pi=x^\perp$: the sets $T_{2,x,y}$ are given by the lines obtained as intersection of the planes $x^\perp$ and $y^\perp$ in $\PG(3,q)$. Hence, every point of $W(3,q)$ is a projective point.
It is also not too hard to show that the points of \(\SCHex(q)\) are projective, but many interesting questions related to the existence of projective points and characterisations of weak generalised polygons based on the existence of such points are wide open. 
\end{remark}

\section{Codes from weak generalised polygons}
The methods used in the proofs of this Section can be traced back to the proof of the characterisation of the code words of minimum weight in the code of points and lines in a projective plane. In particular, as is done in \cite[Theorem 6.3.1]{AssmusAndKey} we will also investigate those vectors that have a constant intersection with the incidence vectors of lines (see Lemma \ref{cx1}).

\subsection{A weighted incidence vector}
The vector space \(\FP\) admits the natural product operation \[\langle \cdot, \cdot \rangle: \FP\times\FP \to \F :(\f,\g)\to \langle \f, \g \rangle := \sum_{x\in\Points}f(x)g(x),\] which is clearly bilinear, symmetric. Using $\f\cdot \g$ for the entry-wise product, we see that $\langle \cdot,\cdot\rangle$ satisfies the property that \[\forall \mathfrak{f}, \mathfrak{g} \in \FP: \langle \mathfrak{f}, \mathfrak{g}\rangle  = \langle \f\cdot\g, \mathbf{1}\rangle,\] where \(\mathbf{1}\) is the all-one vector.

Furthermore, we see that for subsets $X,Y$ of $\Points$, we have
 \(\text{supp}(\i_X) = X\),
		 \(\f = \f\cdot\i_{\text{supp}(\f)}\)
		and \(\i_X\cdot \i_Y = \i_{X\cap Y}\).

For every point $v$ of $\Gamma$, we will now define a particular vector $\cv$.
\begin{definition}\label{cxDef}
	Let \(v\) be a point of \(\Gamma\). 
	
	We can define the following element of \(\FP\) determined by \(v\):
	\[\cv := \sum_{k=0}^{m-1} \left(\sum_{j=0}^{m-k-1} (-s)^{j}\right)\mathfrak{i}_{\Points_{2k}(v)}. \]
	The support of this vector is contained in \(\Points_{\le 2m-2}(v)\). 
\end{definition}

\begin{remark}
We see that in a specific setting, such as a weak generalised hexagon, the expression from the previous definition becomes: 
\[\cv := (s^2-s+1)\i_{\Points_0(v)} + (-s+1)\i_{\Points_2(v)}+ \i_{\Points_4(v)}. \]
Recall that the value \(s^2-s+1\) or \(-s+1\) could evaluate to \(0\) over \(\F\) (see Remark \ref{possibleZero}), but if this does not happen, the support of $\cv$ is precisely $\Points_{\le 2m-2}(v)$
\end{remark}

\begin{lemma}\label{cx1} For all points $v$ and all lines $L$ in $\Gamma$ we have
    \(\langle\cv, \mathfrak{i}_L\rangle = 1\).
\end{lemma}
\begin{proof}
	Let \(L\) be a line of \(\Gamma\). By Corollary \ref{uniqueClosest}, there exists a unique point \(w\) on \(L\) closest to \(v\). 
	Let \(d = \delta(v,w)\). Observe that $d$ is even, that \(d \le 2m-2\), and that 
	\begin{align*}
		\langle\cv, \mathfrak{i}_L\rangle 
		&= \sum_{\{z\in\Points\mid z\I L\}} \cv(z)\\
	\end{align*}

	If \(d = 2m-2\) then we find
	\begin{align*}
		\langle\cv, \mathfrak{i}_L\rangle 
		&= \cv(w) \\
		&= \sum_{j=0}^{0} (-s)^{j} \\
		&= 1.
	\end{align*}
	
	If \(d<2m-2\) then
	\begin{align*}
		\langle\cv, \mathfrak{i}_L\rangle 
		&= \left(\sum_{j=0}^{m-1-\frac{d}{2}} (-s)^{j}\right) + s\left(\sum_{j=0}^{m-1-\frac{d}{2}-1} (-s)^{j}\right) \\
		&= \left(\sum_{j=1}^{m-1-\frac{d}{2}} (-s)^{j}\right) + 1 - \left(\sum_{j=0}^{m-2-\frac{d}{2}} (-s)^{j+1}\right) \\
		&= \left(\sum_{j=1}^{m-1-\frac{d}{2}} (-s)^{j}\right) + 1 - \left(\sum_{j'=1}^{m-1-\frac{d}{2}} (-s)^{j'}\right) \\
		&= 1.
	\end{align*}
\end{proof}

\begin{lemma}\label{cxcy}
	Let \(v\) and \(w\) be two points of \(\Gamma\). 
	Then, \(\cv-\c_w \in \mathfrak{C}_{\mathbb{F}}^D\). 
\end{lemma}
\begin{proof} Let $L$ be a line.	Using Lemma \ref{cx1} we see that
	\begin{align*}
		\langle\cv-\c_w, \mathfrak{i}_L\rangle  =  \langle\cv, \mathfrak{i}_L\rangle - \langle\c_w, \mathfrak{i}_L\rangle = 1 - 1 = 0.     
	\end{align*}
    Since $\mathfrak{C}_\F$ is generated by the vectors $\mathfrak{i}_L$, where $L\in \mathcal{L}$, the statement follows.
\end{proof}

\begin{corollary} \label{AllEq}Let \(\mathfrak{c}\in \Code\) and let  \( v, w \) be points of \(\Gamma\). Then \( \langle \c, \cv \rangle = \langle \c, \c_w \rangle\).
\end{corollary}
\begin{proof}
    Lemma \ref{cxcy} shows that \( \langle \c, \cv - \c_w \rangle = 0\) for all \(\mathfrak{c}\in \Code\).
\end{proof}

\begin{lemma}\label{existsEmpty}
Let \(\c\in \FP\) and let \(C\) be the support of \(\c\).
If \(|C|<s+1\) then
\[\exists v \in \Points: \langle \c, \cv\rangle = 0.\]
\end{lemma}
\begin{proof}
	By Lemma \ref{smallC} we know that there exists a point \(v\) such that 
	\[|\Points_{\le 2m-2}(v)\cap C| = 0.\]
	Since the support of \(\cv\) is contained in \(\Points_{\le 2m-2}(v)\), the statement follows.
\end{proof}

\begin{definition}
	Let \(L\) be a line and \(\f\) an element of \(\FP\).
	We say that \(L\) is \textbf{covered} by \(\f\) if for each point \(v\) on \(L\) we have \(\f(v) \neq 0\). Equivalently, $L$ is covered by $\f$ if and only if \(|\text{supp}(\i_L\cdot\f)| = s+1\).
\end{definition}

\begin{lemma}\label{smallcx}  
	Let \(\c\in \FP\) and let \(C\) be the support of \(\c\). If
	\[\forall L \in \Lines: \exists v \in \Points: L \text{ is covered by } \cv \text{ and } \langle \c, \cv \cdot \i_L \rangle = 0\] then \(|C|> t + 1\).
\end{lemma}
\begin{proof}
	Let \(w\) be a point of \(C\) and assume that \(|C|\le t+1\).
	It follows that at least one of the \(t+1\) lines through \(w\), say $M$, does not contain a second point of \(C\). Our assumption implies that there exists a point \(v\in\Points\) such that \(|\text{supp}(\cv\cdot \i_M)| = s+1 \text{ and } \langle \c, \cv\cdot\i_M \rangle = 0 \). Since \(\langle \c, \cv\cdot\i_M \rangle = \langle \c, \cv\cdot\i_{\{w\}} \rangle = \c(w)\cdot \cv(w)\) where both \(\c(w)\) and \(\cx(w)\) are elements of \(\F\) different from \(0\), this is a contradiction.
\end{proof}

\begin{lemma}\label{smallc}
	Let \(\mathfrak{c}\in \FP\) and let \(C\) be the support of \(\c\).
	If \(|C|<s+1\) then
	\[\forall L \in \Lines:\exists v \in \Points_{\le2m-3}(L): \c\cdot\cv = \c\cdot\cv\cdot\i_L.\]
\end{lemma}
\begin{proof}
	By Lemma \(\ref{starLemma}\) (ii) with \(d=m-1\) we get:
	\[\forall L \in \Lines: \exists v \in \Points_{\le 2m-3}(L): C\cap \Points_{\le2m-2}(v) = C\cap \Points_{1}(L).\]
	Therefore, \[\i_{C\cap \Points_{\le2m-2}(v)}= \i_{C\cap \Points_{1}(L)}.\]
	Which implies that
	\[\c\cdot\cv = \c\cdot\i_C\cdot\cv\cdot\i_{\Points_{\le2m-2}(v)} = \c\cdot\c_v\cdot\i_{C\cap \Points_{\le2m-2}(v)}= \c\cdot\cv\cdot \i_{C\cap \Points_{1}(L)} = \c\cdot\cv\cdot\i_C\cdot \i_L = \c\cdot\cv\cdot\i_L.\]
\end{proof}

\begin{lemma}\label{alsoLines}
	Let \(\mathfrak{c}\in \FP\) and let \(C\) be the support of \(\c\). Assume that \(|C|=s+1\) and
	let \(L\) be a line containing a point \(v\) of \(C\). Then,  
	\[\exists w \in \Points_{\le 2m-3}(L): \langle \c_w\cdot \i_L ,\c\rangle = \langle \c_w,\c\rangle.\]
\end{lemma}
\begin{proof}
	The set $C\setminus\{v\}$ has size $s$, so we can invoke Lemma \ref{starLemma} with $d=m-1$ to deduce that there exists a \(w\in\Points_{\le 2m-3}(L)\) such that \[(C\setminus \{v\}) \cap \Points_{\le 2m-2}(w) = (C\setminus \{v\}) \cap \Points_{1}(L).\]
	Because \(v\in C\cap\Points_{\le 2m-2}(w)\cap\Points_{1}(L)\), we get
	\[C \cap \Points_{\le 2m-2}(w) = C \cap \Points_{1}(L).\]
	The statement now follows since
	\[\c_w\cdot\c = \c_w\cdot\c\cdot\i_{\Points_{\le 2m-2}(w)}\cdot\i_C = \c_w\cdot\c\cdot\i_{\Points_{1}(L)}\cdot\i_C = \c_w\cdot \c\cdot \i_L.\]
\end{proof}
It will follow from the following lemma that code words of weight $s+1$ in $\Code$ give rise to $\mathcal{X}$-blocking sets as defined in the Section \ref{section:BlockingSet}.

\begin{lemma}\label{neverEmpty} Let $s\leq t$.
	Let \(\c\in \Code\) and let \(C\) be the support of \(\c\). Assume that \(|C|=s+1\) and that  
	\[\forall v\in \Points: \text{supp}(\c_v) = \Points_{\leq 2m-2}(v)\]
	Then,
	\[\forall v\in \Points: \langle \c_v ,\c\rangle \neq 0.\]
\end{lemma}
\begin{proof}
	Assume that there exists a point \(v\in \Points\) such that \(\langle \cv ,\c\rangle = 0\). By Corollary \ref{AllEq}, we then have \[\forall w\in \Points: \langle \c_w ,\c\rangle = 0.\] Now let \(v\) be a point of \(C\). Since we assume \(s\le t\), and \(|C\setminus\{v\}| = s\), at least one of the lines through \(v\) does not contain a second point \(C\). Let \(L\) be such a line. By Lemma \ref{alsoLines}, we find that 
	\[\exists w \in \Points_{\le 2m-3}(L): \langle \c_w\cdot \i_L ,\c\rangle = \langle \c_w,\c\rangle = 0.\] 
    
	This means that 
	\[\c_w(v)\cdot\c(v) = 0.\]
	Because \(\F\) is a field and both \(\c_w(v)\) and \(\c(v)\) are assumed to be non-zero, this forms a contradiction.
\end{proof}

\subsection{Proof of the main result}
\begin{theorem}\label{main}
	Let \(\Gamma\) be a finite thick generalised \(2m\)-gon of order \((s,t)\) with \(s\le t\) and \(m>1\). 
	Let \(\F\) be a field such that none of the $m-1$ elements in 
    \[\left\{\sum_{j=0}^{k} (-s)^{j} \mid  1\le k \leq m-1, k\in \Z\right\}\]
    are \(0\) over \(\F\). 
	Then the minimum weight of \(\Code(\Gamma)\) is \(s+1\) and the support of any minimum weight code word is a distance $d$-trace $T_{d,x,y}$ for some $d\in \{1,\ldots,m-1\}$.

    Furthermore, if $s<t$, then $d$ is odd. In particular, if $\Gamma$ is a thick generalised quadrangle of order $(s,t)$ with $s<t$, then every minimum weight code word of $\Code$ is a scalar multiple of the incidence vector of a line.
\end{theorem}
\begin{proof}
	Let \(\mathfrak{c}\in\Code\)  be a codeword with support $C$ and assume to the contrary that $|C|<s
    +1$. 
	
	By Lemma \ref{existsEmpty}, 
	\[\exists v \in \Points: \langle \c, \cv \rangle = 0,\]
	while Corollary \ref{AllEq} implies that 
	\[\forall v, w \in \Points: \langle \c, \cv \rangle = \langle \c, \c_w\rangle. \]
	We conclude that
	\[\forall v\in \Points: \langle \c, \cv\rangle = 0.\]
	From Lemma \ref{smallc} we get
	\[\forall L \in \Lines:\exists v \in \Points_{\le2m-3}(L): \c\cdot\cv = \c\cdot\cv\cdot\i_L.\]
	Combining this, we find 
	\[\forall L \in \Lines:\exists v \in \Gamma_{\le2m-3}(L): \langle \c,\cv\rangle = \langle \c\cdot \cv, \mathbf{1} \rangle = \langle \c\cdot \cv\cdot \i_L, \mathbf{1} \rangle = \langle \c,\cv\cdot \i_L\rangle = 0.\]
	The condition on \(\F\) implies that for all $v\in \Points$, $supp(\cv)=\Points_{\leq 2m-2}(v)$. Since all points on a line that lies at distance at most $2m-3$ from a point $v$ lie at distance at most $2m-2$ of $v$, we see that all lines at distance at most \(2m-3\) from \(v\) are covered by \(\cv\) i.e.:
	\[\forall L\in \Lines: \forall v \in \Points_{\le 2m-3}(L): |\text{supp}(\i_L\cdot\cv)| = s+1.\]
	Therefore, the conditions for Lemma \ref{smallcx} are fulfilled, which gives us \(|C|>t+1\). This contradicts the assumptions that $|C|<s+1$ and \(s\le t\).
	
	Since for any line \(L\in \Lines\), the codeword \(\mathfrak{i}_L\) has weight \(s+1\), we have shown that the minimum weight of $\Code$ is $s+1$.

    Now assume that $|C|=s+1$. The condition on $\F$ shows that we can invoke Lemma \ref{neverEmpty} to find that $C$ forms an $\mathcal{X}$-blocking set.  Theorem \ref{distanceTraces} now shows that $C$ is a distance trace $T_{d,x,y}$. Furthermore, if $s<t$, then $d$ is odd. If additionally $\Gamma$ is a weak generalised quadrangle, then $m=2$, so since $d\in \{1,\ldots,m\}$ is odd, only the possibility $T_{1,x,y}$ occurs: the support $C$ is the point set of a line. 
    Now assume that $C$ is the point set of a line $L$. Let $v$ be a point of $L$. Consider the vector $\c-\c(v)\mathfrak{i}_L$. This is a linear combination of code words in $\Code$, so it is a code word of $\Code$, but its weight is at most $s$. Therefore, it is the zero vector. It follows that $\c=\c(v)\mathfrak{i}_L$. 
\end{proof}

\begin{remark}
	We could state Theorem \ref{main} for finite weak generalised \(2m\)-gons, but the requirement \(1-s\neq 0\), together with \(s\le t\) forces the \(2m\)-gon to be thick since this implies that every point and every line is incident with at least three other elements (see Remark \ref{thickAlt}).
\end{remark}

\begin{remark}
	The condition on the field $\F$ in Theorem \ref{main} boils down to the following: for finite thick generalised quadrangles, the condition \(s\neq 1\) in $\F$, is automatically fulfilled by being thick, for finite thick generalised hexagons, we have the additional requirement \(s^2-s+1\neq 0\) in $\F$, and for finite thick generalised octagons we require \(0\notin \{-s^3+s^2-s+1, s^2-s+1\}\) in $\F$. 
    In particular, the classical case, where we study the $p$-ary code of a polygon embedded in $\PG(n,q)$, $q=p^h$, $p$ prime, is included in the theorem since in that case, $s=q=0$ in \(\F=\F_p\).
\end{remark}

\begin{remark} After presenting this work, we became aware of current ongoing research  that is expected to contain a slightly stronger version of Theorem \ref{distanceTraces} for the case of generalised hexagons \cite{SiraHendrik}. In this paper, which is currently still in preparation, the condition $s\leq t$ would not be present.
\end{remark}

\paragraph{Acknowledgement}
The authors would like to thank the anonymous reviewers for their helpful suggestions, in particular regarding Definition \ref{cxDef}.

\bibliographystyle{abbrv}
\bibliography{sources}
\end{document}